\documentclass[12pt]{article}
\usepackage{latexsym}
\usepackage{amsmath, amsthm}
\usepackage[all,dvips,ps]{xy}
\usepackage{amsfonts}
\usepackage{mathrsfs}

\newtheorem{lem}{Lemma}[section]
\newtheorem{thm}{Theorem}[section]

\newtheorem{exa}{Example}[section]

\newcommand{\diag}{\mathrm{diag}\,}
\newcommand{\rank}{\mathrm{rank}\,}

\newcommand{\Rs}{\mathbb{R}}

\newcommand{\bz}{{\bf 0} }

\newcommand{\lamx}{\lambda_{\max}}

\newcommand{\beq}{ \begin{equation} }
\newcommand{\eeq}{ \end{equation} }
\newcommand{\bt}{ \begin{tabular} }
\newcommand{\et}{ \end{tabular} }

\begin{document}

\bibliographystyle{plain}
\title{On Theorems of \v{S}i\v{n}ajov\'{a}, Rankin and Kuperberg Concerning Spherical Point Configurations}
\vspace{0.3in}
        \author{ A. Y. Alfakih
  \thanks{E-mail: alfakih@uwindsor.ca}
  \\
          Department of Mathematics and Statistics \\
          University of Windsor \\
          Windsor, Ontario N9B 3P4 \\
          Canada
}

\date{August 2, 2019. Revised \today}
\maketitle

\noindent {\bf AMS classification:} 15A18, 51K05, 52C17.

\noindent {\bf Keywords:} Orthogonal representation of graphs, distance geometry, sphere packing, dispersion problem, Euclidean distance matrices.
\vspace{0.1in}

\begin{abstract}
This note presents simple linear algebraic proofs of theorems
due to  \v{S}i\v{n}ajov\'{a}, Rankin and Kuperberg concerning spherical point configurations.
The common ingredient in these proofs is the use of spherical Euclidean distance matrices and the Perron-Frobenius theorem.
\end{abstract}

\section{Introduction}

In this note, we are interested in four theorems on spherical point configurations.
The first of these theorems is concerned with orthonormal representation of graphs.
The notion of orthonormal representation of a graph was introduced by Lov\'{a}sz in his study of Shannon capacity of graphs \cite{lov79}.
For a detailed discussion of orthonormal representations see the recent book \cite{lov19}.
Parsons and Pisanski \cite{pp89} introduced the following notion of orthonormal representation, which is slightly different from that of Lov\'{a}sz
\footnote{In Lov\'{a}sz's definition,
the inner product $(p^i)^Tp^j$ is unrestricted if nodes $i$ and $j$ are adjacent.}.
Let $G$ be a simple graph with nodes $1,\ldots,n$. An {\em orthonormal representation } of $G$ is a mapping of the nodes of $G$ to
 unit vectors  $p^1,\ldots,p^n$  in Euclidean $r$-space $\Rs^r$ such that $(p^i)^Tp^j$ is negative or zero
depending on whether nodes $i$ and $j$ are adjacent or not.
The smallest dimension $r$ necessary for such a representation is denoted by $d(G)$. It is easy to see that
$d(G) \geq \alpha(G)$, where $\alpha(G)$ is the independence number of $G$.
 \v{S}i\v{n}ajov\'{a} proved the following.

\begin{thm}[\v{S}i\v{n}ajov\'{a} \cite{sin91}] \label{thmsin}
Let $G$ be a simple graph on $n$ nodes and let $k$ be the number of its nontrivial connected components, i.e., those connected components
with at least $2$ nodes. Then
$d(G) = n - k$.
\end{thm}

The remaining theorems are concerned with the dispersion problem.
The {\em dispersion problem} is the problem of maximizing, over all $n$-point configurations
on the unit ($r-1$)-sphere in $\Rs^r$, the minimum distance between any two points.
The dispersion problem has applications in sphere packing and spherical designs \cite{wyx18}.
Davenport and Haj\'{o}s \cite{dh51} and Rankin \cite{ran55} provided solutions of this problem for the case $n=r+2$.
Rankin \cite{ran55}, also, provided a solution for the case $n=2r$.
Before presenting Rankin's two theorems, we need the following definition.
The {\em regular $r$-crosspolytope} is the convex hull of the union of $r$ mutually orthogonal line segments of length 2
and intersecting at their common midpoint. That is, the regular $r$-crosspoltope is the convex hull of ($\pm e^1, \pm e^2, \ldots, \pm e^r$),
where $e^i$ is the $i$th standard unit vector in $\Rs^r$.

\begin{thm}[Rankin \cite{ran55}]  \label{thmran1}
 Let $p$ be an $n$-point configuration on the unit ($r-1$)-sphere in $\Rs^r$. If $n=r+2$, then two points
of $p$ are at a distance of at most $\sqrt{2}$ from each other.
\end{thm}

\begin{thm}[Rankin \cite{ran55}] \label{thmran2}
Let $p$ be an $n$-point configuration on the unit ($r-1$)-sphere in $\Rs^r$. If $n=2r$ and the distance
between any two points of $p$ is $\geq \sqrt{2}$, then $p$ is unique, up to a rigid motion, and the points  of $p$ are the vertices
of the regular $r$-crosspolytope.
\end{thm}

Kuperberg \cite{kup07}  generalized Rankin's result to all $n$:  $r+2 \leq n \leq 2r$.

\begin{thm}[Kuperberg \cite{kup07}] \label{thmkup1}
Let $p$ be an $n$-point configuration on the unit sphere in $\Rs^r$ such that $2 \leq n- r \leq r$.
If the minimum distance between any two points of $p$ is at least $\sqrt{2}$, then
$\Rs^r$ can be split into the orthogonal product $\prod_{i=1}^{n-r} L_i$ of $n-r$ subspaces of $\Rs^r$
such that $L_i$ contains exactly $r_i + 1$ points of $p$, where $r_i$ is the dimension of $L_i$.
\end{thm}

In this note, we present simple linear algebraic proofs of \v{S}i\v{n}ajov\'{a}, Rankin and
 Kuperberg's theorems based on
spherical Euclidean distance matrices (EDMs) and the Perron-Frobenius theorem.
These proofs are given in Sections 3, 4 and 5 respectively, while the necessary background material is given in Section 2.

 \subsection{Notation}
 We collect here the notation used in this note. $e_n$ and $E_n$ denote, respectively, the vector of all 1's in $\Rs^n$
 and the matrix of all 1's of order $n$. $I_n$ denotes the identity matrix of order $n$. $e^i_n$ denotes the $i$th column of $I_n$.
 The subscript $n$, in $e_n$, $E_n$, $I_n$ and $e^i_n$  will be omitted if the dimension is clear from the context.
 For a matrix $A$, we denote the vector consisting of the diagonal entries of $A$ by $\diag(A)$. Also,
 for a real symmetric matrix $A$, we denote by $\lamx(A)$ and $m(\lamx(A))$, respectively,
 the largest eigenvalue of $A$ and its multiplicity. The zero vector or the zero matrix of the appropriate dimension
 is denoted by $\bz$.
 PSD stands for positive semidefinite.
 Finally, $E(G)$ denotes the edge set of a simple graph $G$.

\section{Preliminaries}

In this section we present the necessary background concerning EDMs and more specifically spherical EDMs.
For a comprehensive treatment of EDMs see the monograph \cite{alfm18}.

 An $n \times n$ matrix $D=(d_{ij})$ is said to be an EDM
 if there exist points $p^1,\ldots,p^n$ in some Euclidean space such that
\[
d_{ij}= || p^i - p^j ||^2 \mbox{ for all } i,j=1,\ldots, n,
\]
where $||x||$ denotes the Euclidean norm of $x$, i.e., $||x|| = \sqrt{x^Tx}$.
$p^1,\ldots,p^n$ are called the {\em generating points} of $D$ and the dimension of their affine span is
called the {\em embedding dimension} of $D$.
 If the embedding dimension of an $n \times n$ EDM $D$ is $n-1$, then we refer to $D$ as the EDM of a simplex.
For example, let $E$ and $I$ denote respectively the matrix of all 1's and the identity matrix. Then
the EDM $D = \gamma(E - I)$, where $\gamma$ is a positive scalar, is
 the EDM of a {\em regular simplex}. An EDM $D$ is said to be {\em spherical} if its generating points lie on a sphere.
A {\em unit} spherical EDM is a spherical EDM whose generating points lie on a sphere of radius $\rho = 1$.

Let $e$ denote the vector of all 1's in $\Rs^n$ and let $s$ be a vector in $\Rs^n$ such that $e^Ts= 1$.
The following theorem is a well-known characterization of EDMs \cite{sch35,yh38, gow85,cri88}.
\begin{thm} \label{thmEDM}
Let $D$ be an $n \times n$ real symmetric matrix whose diagonal entries are all 0's. Then $D$ is an EDM
if and only if
\beq  \label{defBs}
B= -\frac{1}{2} (I - e s^T) D (I - s e^T )
\eeq
is positive semidefinite (PSD), in which case, the embedding dimension of $D$ is given by rank($B$).
\end{thm}
That is, $D$ is an EDM iff it is negative semidefinite on $e^{\perp}$, the orthogonal complement of $e$ in $\Rs^n$.
It can be easily shown that $B$ as defined in Equation (\ref{defBs}) is a Gram matrix of the generating points of $D$,
or a Gram matrix of $D$ for short.

Let $B$ be a Gram matrix of an EDM $D$ with rank $r$. Then $B$ is PSD and hence $B=PP^T$ for some $n \times r$ matrix $P$. Consequently,
$p^1, \ldots, p^n$, the generating points of $D$, are given by the rows of $P$.
As a result, $P$ is called a {\em configuration matrix} of $D$.
It should be noted that Equation (\ref{defBs}) implies that $Bs=\bz$ and hence $P^Ts=\bz$; that is
\beq \label{eqnsps}
\sum_{i=1}^n s_i p^i = \bz.
\eeq

It is well known \cite{gow85} that if $D$ is a nonzero EDM, then $e$ lies in the column space of $D$, i.e.,
there exists vector $w$ such that
\beq \label{defw}
Dw = e.
\eeq
It is also well known that if $D$ is an $n \times n$ EDM of a simplex, i.e., if the embedding dimension of $D$ is $n-1$,
then $D$ is spherical and nonsingular.
Among the many different characterizations of spherical EDMs, the one
given in the following theorem is the most relevant for our purpose.

\begin{thm}[\cite{gow82,gow85}]  \label{thmgow}
Let $D$ be an EDM and let $Dw =e$. Then $D$ is spherical if and only if
$e^Tw >0$, in which case, $\rho$, the radius of the sphere containing the generating points of $D$, is given by
\beq
\rho = \left( \frac{1}{2 e^T w } \right)^{1/2}.
\eeq
\end{thm}

As an example consider $D= \gamma (E_n-I_n)$, the EDM of a regular simplex. Then
$w= e / (\gamma (n-1))$ and thus its generating points lie on a sphere of radius
$\rho = \sqrt{\gamma (n-1)/(2n)}$. Consequently, the $n \times n$ unit spherical EDM of a regular simplex is given by
\[
  D = \frac{2n}{n-1} (E_n - I_n).
\]

A vector $x$ is {\em positive}, denoted by $x > \bz$, if each of its entries is positive.
Similarly, a matrix $A$ is {\em positive} ({\em nonnegative}), denoted by $A > \bz$ ($A \geq \bz$), if each of its entries
is $> 0$ ($\geq 0$).
An $n \times n$ nonnegative matrix $A$ is said to be {\em reducible} if $A$ is the $1 \times 1$ zero matrix or
 if $n \geq 2$ and there exists a permutation matrix $Q$ such that
\[
 Q A Q^T = \left[ \begin{array}{cc} A_{11} & A_{12} \\
                                     \bz & A_{22} \end{array} \right],
\]
where $A_{11}$ and $A_{22}$ are square matrices.
It easily follows from the definition that if  $A$ is a  nonnegative symmetric reducible matrix of order $n \geq 2$,
then there exists a permutation matrix $Q$ such that
$Q A Q^T$ is a block diagonal matrix, of at least two blocks, such that each block is either irreducible or the $1 \times 1$ zero matrix.
A nonnegative matrix that is not reducible is {\em irreducible}.

It is well known that an $n \times n$ nonnegative matrix $A$ is irreducible if and only if $(I + A)^{n-1}> \bz$.
Moreover, if $A$ is the adjacency matrix of a simple graph $G$, then $A$ is irreducible if and only if $G$ is connected.
We will need the following fact from the celebrated Perron-Frobenius theorem: If $A$ is a nonnegative irreducible matrix, then
the largest eigenvalue of $A$, $\lamx(A)$, is positive with multiplicity $m(\lambda_{\max}(A)) = 1$ and
the eigenvector associated with $\lamx(A)$ is positive.

\section{Proof of \v{S}i\v{n}ajov\'{a} Theorem}

A connected component of a graph $G$ is said to be {\em nontrivial} if it consists of at least 2 nodes.
In other words, isolated nodes are trivial connected components of $G$.
Now let $p^i$ and $p^j$ be two unit vectors. Then, clearly, $(p^i)^Tp^j = 0$ if and only if  $||p^i-p^j||^2 = 2$ and
$(p^i)^Tp^j < 0$ if and only if  $||p^i-p^j||^2 > 2$. As a result,
Theorem \ref{thmsin} can be stated in the language of EDMs as follows.

\begin{thm}[\v{S}i\v{n}ajov\'{a} \cite{sin91}] \label{thmsin2}
Let $G$ be a simple graph on $n$ nodes and let $k$ be the number of its nontrivial connected components.
Then there exists a unit spherical EDM $D=(d_{ij})$ of embedding dimension $r= n-k$ such that
\beq \label{eqdijsin}
d_{ij}  \left\{ \begin{array}{ll} > 2  \mbox{ iff } \{i,j\} \in E(G), \\
                                  = 2  \mbox{ iff } \{i,j\} \not \in E(G),  \end{array} \right.
\eeq
where $E(G)$ denotes the edge set of $G$.
Furthermore, there does not exist a unit spherical EDM of embedding dimension $r \leq n-k-1$ that satisfies (\ref{eqdijsin}).
\end{thm}

Before proving Theorem \ref{thmsin2}, we first prove the following lemma.

\begin{lem} \label{lemBDel}
Let $D$ be an $n \times n$ unit spherical EDM of embedding dimension $r$ and let $Dw=e$.
Let $D = 2 (E-I) + 2 \Delta$. Then
$\lamx (\Delta) = 1$ and $w$ is an eigenvector associated with $\lamx(\Delta)$. Moreover,
$r=n - m(\lamx(\Delta))$, where $m(\lamx(\Delta))$ denotes the multiplicity of $\lamx(\Delta)$.
\end{lem}
\begin{proof}
By Theorem \ref{thmgow}, $2 e^T w = 1$. Thus by setting $s = 2 w$ in Equation (\ref{defBs}),
it follows that the corresponding Gram matrix of $D$ is
\beq
  B = E - \frac{1}{2}D = I - \Delta.
\eeq
Hence $\lamx(\Delta) \leq 1$ since $B$ is PSD.
On the other hand, $Bw=\bz$ implies that
\beq
\Delta w = w.
\eeq
Hence $\lamx(\Delta) \geq 1$ and consequently $\lamx(\Delta) = 1$.
As a result, $r = \rank (B)= n - m(\lamx(\Delta))$.
\end{proof}

Now we are ready to prove Theorem \ref{thmsin2}.

\begin{proof}[Proof of Theorem \ref{thmsin2}]
Let $A$ denote the adjacency matrix of $G$.
Then there exists a permutation matrix $Q$ such that
\beq \label{QAQT}
Q A Q^T = \left[ \begin{array}{cccc} A^1 &  &   & \\
                                        &   \ddots & & \\
                                      &  &   A^{k} & \\
                                      & & & \bz \end{array} \right],
\eeq
where $A^1, \ldots, A^k$ denote the adjacency matrices of the nontrivial connected components of $G$.
Hence, $A^1, \ldots, A^k$ are irreducible nonnegative matrices of orders $\geq 2$.
Therefore, by the Perron-Frobenius theorem, $m(\lamx(A^1))=\cdots = m(\lamx(A^k)) =1$.
For $i=1,\ldots,k$, let $\xi^i$ denote the eigenvector of $A^i$ associated with $\lamx(A^i)$ and
let $\Delta^i = A^i/\lamx(A^i)$. Further,
let
\[
 \Delta = \left[ \begin{array}{cccc} \Delta^1 &  &   & \\
                                        &   \ddots & & \\
                                      &  &   \Delta^{k} & \\
                                      & & & \bz \end{array} \right], \;\;
 \xi = \left[ \begin{array}{c} \xi^1  \\ \vdots \\ \xi^{k} \\ \bz \end{array} \right] \mbox{ and } w = \frac{\xi}{2 e^T \xi}.
\]
Then, obviously, $\Delta_{ij} > 0$  if and only if $\{i,j\} \in E(G)$ and $\Delta_{ij} = 0$  if and only if
 $i=j$ or  $\{i,j\} \not \in E(G)$. Also, it is equally obvious that
$\lamx(\Delta)=1$, $m(\lamx(\Delta))=k$ and $\Delta w = w$.

Let $D = 2(E-I)+2 \Delta$. Then $Dw= e$ since $2 e^Tw = 1$. Now if we let $s=2 w$ in Equation (\ref{defBs}), then
\[
B = -\frac{1}{2} (I-e s^T) D (I-se^T) =  E - \frac{1}{2} D = I - \Delta
\]
is PSD and of rank $n-k$. As a result, by Theorems \ref{thmEDM} and \ref{thmgow},
$D$ is a unit spherical EDM of embedding dimension $r=n-k$ that satisfies (\ref{eqdijsin}).

To complete the proof, let $r$ be the embedding dimension of any unit spherical EDM $D$ that satisfies (\ref{eqdijsin}).
Let $\Delta = D/2+I-E$ and wlog assume that $\Delta$ is block diagonal. Thus $\Delta$ has $k$ irreducible nonnegative diagonal blocks,
each associated with a nontrivial connected component of $G$. Now it follows from Lemma \ref{lemBDel} that
$\lamx(\Delta) = 1$ and $r = n - m(\lamx(\Delta))$. Consequently, $r \leq n-k$ since the contribution from each irreducible
diagonal block of $\Delta$ to $m(\lamx(\Delta))$ is at most 1.
\end{proof}

\section{Proof of Rankin's Theorems}

Theorems \ref{thmran1} and \ref{thmran2} can be stated in the language of EDMs as follows.

\begin{thm}[Rankin \cite{ran55}] \label{thmran12}
Let $D$ be an $n \times n$ unit spherical EDM of embedding dimension $r$.
If $n=r+2$, then at least one off-diagonal entry of $D$ is  $\leq 2$.
\end{thm}

\begin{thm}[Rankin \cite{ran55}] \label{thmran22}
Let $D$ be an $n \times n$ unit spherical EDM of embedding dimension $r$.
If $n=2r$ and if each off-diagonal entry of $D$ is $\geq 2$, then there exists a permutation matrix $Q$ such that
\beq \label{eqnEDMcp}
QDQ^T = \left[ \begin{array}{cccc} 4(E_2-I_2) & 2E_2 & \cdots & 2E_2 \\
                                     2 E_2 & 4(E_2-I_2) & \cdots & 2E_2  \\
                                     \vdots & \vdots &  \ddots & \vdots \\
                                     2E_2 & \cdots & 2E_2 &   4(E_2-I_2) \end{array} \right],
\eeq
where $E_2$, $I_2$ are, respectively, the matrix of all 1's and the identity matrix of orders $2$.
\end{thm}

It should be noted that the RHS of Equation (\ref{eqnEDMcp}) is the EDM of the regular $r$-crosspolytope.
As was mentioned earlier, Theorems \ref{thmran1} and \ref{thmran2} are special cases of Theorem \ref{thmkup1} which we prove
in the next section. However, in this section, we present an independent proof of Theorem \ref{thmran1} after we have proved
the following lemma which will be needed in the sequel.

\begin{lem} \label{lemsimplex}
Let $D$ be an $n \times n$ unit spherical EDM of embedding dimension $r$ and assume that each off-diagonal entry of $D$ is $\geq 2$.
Let $D = 2 (E-I) + 2 \Delta$ and let $Dw=e$.
If $\Delta$ is irreducible,  then $r = n-1$, i.e.,
$D$ is the EDM of a simplex, and $w > \bz$.
\end{lem}

\begin{proof}
Clearly, $\Delta \geq \bz$. Thus, it follows from Lemma \ref{lemBDel} and the Perron-Frobenius theorem
that $\lamx(\Delta)=1$, $m(\lamx(\Delta))= 1$ and  $w > \bz$.
Consequently, $r = \rank(B) = n-1$.
\end{proof}

Now Theorem \ref{thmran1} is an immediate corollary of Lemma \ref{lemsimplex}.

\begin{proof}[Proof of Theorem \ref{thmran12}]
Let $\Delta = D/2 + I - E$ and
assume, by way of contradiction, that each off-diagonal entry of $D$ is $> 2$. Then each off-diagonal entry of $\Delta$ is $> 0$.
Hence, $I + \Delta > 0$ and thus $\Delta$ is irreducible. Consequently, by Lemma \ref{lemsimplex},
the embedding dimension of $D$ is $r=n-1$, which contradicts the assumption that $r=n-2$.
\end{proof}

\section{Proof of Kuperberg's Theorem}

Theorem \ref{thmkup1} can be stated in the language of EDMs as follows.

\begin{thm}[Kuperberg \cite{kup07}]  \label{thmkup2}
Let $D$ be an $n \times n$ unit spherical EDM of embedding dimension $r$,
where $2 \leq n- r \leq r$. If each off-diagonal entry of $D$ is $\geq 2$,
then there exists a permutation matrix $Q$ such that
\[
Q D Q^T = \left[ \begin{array}{cccc} D^1 & 2E & \cdots & 2E \\
                                     2E   & D^2  & \cdots & 2E \\
                                     \vdots & \vdots &  \ddots & \vdots \\
                                     2E & \cdots & 2E &  D^{n-r} \end{array} \right],
\]
where $D^1, \ldots, D^{n-r}$ are unit spherical EDMs of simplices; and
$E$ is the matrix of all 1's of the appropriate dimension.
\end{thm}

Two remarks are in order here.
First, as shown in \cite{kup07}, if $n=r+2$, then Theorem \ref{thmkup2} reduces to Rankin's Theorem \ref{thmran12}. This follows since
 if $D$ has an off-diagonal entry $< 2$, then there is nothing to prove. On the other hand, if every off-diagonal entry of $D$
is $\geq 2$, then Theorem \ref{thmkup2} implies that there is a permutation matrix $Q$ such that
$Q D Q^T = \left[ \begin{array}{cc} D^1 & 2E \\ 2E & D^2 \end{array} \right]$. Hence, at least one of the off-diagonal diagonal entries of $D$
is $2$  since $2E$ is a submatrix of $Q D Q^T$.

Second, also, as shown in \cite{kup07}, if $n=2r$, i.e., if $n-r = r$, then Theorem \ref{thmkup2} reduces to Rankin's Theorem \ref{thmran22}.
This follows since in this case, each of the submatrices $D^1,\ldots,D^r$ in Theorem \ref{thmkup2} is of order $2$, and thus
$D^1=\cdots=D^r= 4(E_2-I_2)$.  Therefore,
the configuration, in this case,  is that of the regular $r$-crosspolytope since
the matrix $Q D Q^T$ in Theorem \ref{thmkup2} reduces to that in Theorem \ref{thmran22}.

Before presenting the proof of Theorem \ref{thmkup2}, we need the following lemma which extends Lemma \ref{lemsimplex}
to the case where $\Delta$ is padded with zero rows and columns.

\begin{lem} \label{lemsimplex2}
Let $D$ be an $n \times n$ unit spherical EDM of embedding dimension $r$ and assume that  each off-diagonal entry of $D$ is $\geq 2$.
Let $ D = 2(E-I) + 2 \tilde{\Delta}$ and let $D \tilde{w} =e$.
If $\tilde{\Delta} = \left[ \begin{array}{cc} \Delta & \bz \\ \bz & \bz \end{array} \right]$,
where  $\Delta$ is irreducible, then $r=n-1$,
i.e., $D$ is the EDM of a simplex, and
$\tilde{w} = \frac{1}{2 e^T \xi} \left[ \begin{array}{c} \xi \\ \bz \end{array} \right]$, where $\Delta \xi = \xi$ and $\xi > \bz$.
\end{lem}

\begin{proof}
 The proof is similar to that of Lemma \ref{lemsimplex}.
\end{proof}

Now we are ready to prove  Theorem \ref{thmkup2}.

\begin{proof}[Proof of Theorem \ref{thmkup2}]

Let $D= 2(E-I) + 2 \Delta $ and thus $\Delta \geq \bz$ and $\diag(\Delta) = \bz$.
Since the embedding dimension of $D$ is $r$, it follows from Lemma \ref{lemBDel}
 that $\lamx(\Delta) = 1$ with multiplicity $m(\lamx(\Delta))= n-r \geq 2$.
Therefore, by the Perron-Frobenius theorem, $\Delta$ is reducible and thus there exists a permutation matrix $Q$
such that
\beq \label{QDelQT}
Q \Delta Q^T = \left[ \begin{array}{ccc} \Delta^1   &  &  \\
                                           &  \ddots &  \\
                                      &   & \Delta^{n-r} \end{array} \right] \mbox{ or }
 \left[ \begin{array}{cccc} \Delta^1 &  &   & \\
                                        &   \ddots & & \\
                                      &  &   \Delta^{n-r} & \\
                                      & & & \bz \end{array} \right],
\eeq
where $\Delta^1, \ldots, \Delta^{n-r}$ are irreducible and thus $\lamx(\Delta^1) = \cdots = \lamx(\Delta^{n-r})= 1$.
For $i=1,\ldots,n-r$, let $\xi^i$ denote the eigenvector of  $\Delta^i$ associated with $\lamx(\Delta^i)$.
Therefore, by the Perron-Frobenius theorem $\xi^i > \bz$ since $\Delta^i$ is irreducible.
Next, we consider the two cases of $Q \Delta Q^T$ in Equation (\ref{QDelQT}) separately.

In the first case, all diagonal blocks of $\Delta$ are irreducible.
Assume that, for $i=1,\ldots,n-r$, $\Delta^i$ is of order $n_i$ where $\sum_{i=1}^{n-r} n_i = n$. Then
$n_i \geq 2$ since $\diag(\Delta^i)=\bz$.
Let $D^i = 2(E_{n_i} - I_{n_i}) + 2 \Delta^i$ for $i=1, \ldots,n-r$.
Then $D^1, \ldots , D^{n-r}$ are EDMs since they are principal submatrices of $D$.
Moreover, let $w^i =  \xi^i / (2 e^T_{n_i} \xi^i) $. Then $D^i w^i = e_{n_i}$ and $w^i > \bz$. Consequently,
$D^1, \ldots, D^{n-r}$ are unit spherical EDMs.
Therefore,  it follows from Lemma \ref{lemsimplex}
that each of $D^1, \ldots, D^{n-r}$ is the EDM of a simplex.
It is worth pointing out that Equation (\ref{eqnsps}) implies that, for each $i=1,\ldots,n-r$,
the origin $\bz$ lies in the relative interior \cite{mus19} of the convex hull of the generating points of $D^i$  since $w^i > \bz$.

In the second case, let
$\tilde{\Delta}^{n-r}  = \left[ \begin{array}{cc} \Delta^{n-r} & \\ & \bz \end{array} \right]$.
Then, similar to the first case, $D^1, \ldots, D^{n-r-1}$ are unit spherical EDMs of simplices and
the origin $\bz$ lies in the relative interior of the convex hull of the generating points of each of the EDMs $D^1, \ldots, D^{n-r-1}$.
On the other hand, let $D^{n-r} = 2(E-I) + 2 \tilde{\Delta}^{n-r}$ and let
\[
\tilde{w}^{n-r} = \frac{1}{2 e^T \xi^{n-r} } \left[ \begin{array}{c} \xi^{n-r} \\ \bz \end{array} \right].
\]
Then $\tilde{\Delta}^{n-r} \tilde{w}^{n-r} = \tilde{w}^{n-r}$ and $D^{n-r} \tilde{w}^{n-r} = e$.
Hence, $D^{n-r}$ is a unit spherical EDM and hence, by Lemma \ref{lemsimplex2}, $D^{n-r}$ is the EDM of a simplex.
However, unlike $D^1, \ldots, D^{n-r-1}$,
the origin lies on the relative boundary of the convex hull of the generating points of $D^{n-r}$.
\end{proof}

Finally, we should point out that in the second case of Equation (\ref{QDelQT}), i.e., if  $Q \Delta Q^T$ has, say $s$, zero rows (and columns),
then we chose above to define $\tilde{\Delta}^{n-r}$ by appending these $s$ zero rows and columns to $\Delta^{n-r}$.
In fact, we could have appended any number of these zero rows and columns to any of $\Delta^1, \ldots, \Delta^{n-r}$.

As an  illustration of  the theorems of \v{S}i\v{n}ajov\'{a} and Kuperberg, consider the following example.

\begin{exa}
Let $G$ be the simple graph on the nodes $1,\ldots,5$ and with edge set $E(G)=\{ \; \{1,2\}, \{3,4\} \;\}$.
Hence, $G$ has two nontrivial connected components and one isolated node.
To illustrate \v{S}i\v{n}ajov\'{a}'s Theorem, let
\beq
\Delta = A = \left[ \begin{array}{ccccc} 0 & 1 &  &  & \\
                                      1  & 0 & & & \\
                                      &  & 0 & 1 &  \\
                                      & & 1 & 0 &  \\
                                      & & & & 0 \end{array} \right]
= \left[ \begin{array}{ccc} \Delta^1 &    & \\
                                      &    \Delta^{2} & \\
                                      & &  \bz \end{array} \right],
\eeq
where $A$ is the adjacency matrix of $G$. Then $D=2(E-I)+2\Delta$ is a unit spherical EDM of embedding dimension $3$ that satisfies
(\ref{eqdijsin}). Moreover, an orthonormal representation of $G$ is given by
$p^1 = e^1$, $p^2 = - e^1$, $p^3 = e^2$, $p^4 = - e^2$ and $p^5 = e^3$,
where $e^i$ is the $i$th standard unit vector in $\Rs^3$.

To illustrate Kuperberg's Theorem,
first, if we define $\tilde{\Delta}^{2}  = \left[ \begin{array}{cc} \Delta^{2} & \\ & \bz \end{array} \right]$.
Then  $\Rs^3$ can be split into $2$ orthogonal subsapces $L_1$ and $L_2$ where
$L_1$ consists of the $x$-axis and contains points $p^1$ and $p^2$; while $L_2$  consists of the $y$--$z$ plane and contains points $p^3$, $p^4$ and $p^5$.
Notice that the origin is in the relative interior of the convex hull of $p^1$ and $p^2$, while the origin lies on the relative
boundary of the convex hull of $p^3, p^4$ and $p^5$.

On the other hand, if we define $\tilde{\Delta}^{1}  = \left[ \begin{array}{cc} \Delta^{1} & \\ & \bz \end{array} \right]$.
Then, in this case,  the subspace $L_1$ consists of the $x$--$z$ plane and  contains points $p^1,p^2$ and $p^5$,
while $L_2$  consists of the $y$-axis and contains points $p^3$ and $p^4$.
\end{exa}

\noindent{\bf \large Acknowledgements} I would like to thank Marton Nasz\'{o}di for bringing my attention,
after this note was first posted on the arXiv, 
the reference [10], where the Perron-Frobenius theorem is used to prove Rankin's theorem. 

\bibliographystyle{plain}

\end{document}